\documentclass[10pt,letterpaper,twoside]{amsart}
\usepackage{charter, amsmath,amsbsy,amsfonts,amssymb,
stmaryrd,amsthm,mathrsfs,graphicx, amscd, tikz-cd, cancel}
\usepackage{srcltx}

\usepackage[OT2,OT1]{fontenc} \newcommand\cyr
{
\renewcommand\rmdefault{wncyr} \renewcommand\sfdefault{wncyss} \renewcommand\encodingdefault{OT2} \normalfont
\selectfont
}
\DeclareTextFontCommand{\textcyr}{\cyr}    
\makeatletter
   \def\@settitle
{
\begin{center} \baselineskip14\p@\relax \LARGE\@title \end{center}
}
  \makeatother

\usepackage[normalem]{ulem}
\usetikzlibrary{matrix,arrows,decorations.pathmorphing}
\usepackage[
	hypertexnames=false,
	hyperindex,
	pagebackref,
	pdftex,
	breaklinks=true,
	bookmarks=false,
	colorlinks,
	linkcolor=blue,
	citecolor=red,
	urlcolor=red,
]{hyperref}
\usepackage[mathscr]{eucal}
\numberwithin{equation}{section}

\usepackage{verbatim}
\usepackage{amssymb}
\usepackage{amsbsy}
\usepackage{amscd}
\usepackage{amsmath}
\usepackage{amsthm}
\usepackage[mathscr]{eucal}

\def\CC{{\mathbb C}}

\def\GG{{\mathbb G}} 
\def\HH{{\mathbb H}}

\def\PP{{\mathbb P}}
\def\QQ{{\mathbb Q}} 
\def\RR{{\mathbb R}}

\def\ZZ{{\mathbb Z}}

\def\sllie{\mathfrak{sl}}

\def\G{\Gamma}

\def\Ms{\mathscr{M}}

\def\pure{\mathrm{pure}} 
 
\def\ssm{\smallsetminus}

\newcommand{\p}{\partial}

\def\Ccal{{\mathcal C}}

\def\Ncal{{\mathcal N}}
\def\Ocal{{\mathcal O}}

\def\Ss{{\mathscr S}}
\def\Vs{{\mathscr V}}

\def\Xcal{{\mathcal X}} 
 
\def\Zcal{{\mathcal Z}}

\def\half{{\tfrac{1}{2}}}

\def\End{\operatorname{End}}

\def\pt{{\scriptscriptstyle\bullet}}

\newcommand\per{\mathscr{D}}
\newcommand\Per{\mathscr{P}}

\newcommand\Uscr{\mathscr{U}}

\newcommand\aut{\operatorname{Aut}}

\newcommand\diff{\operatorname{Diff}}

\newcommand\Gr{\operatorname{Gr}}
\newcommand\HK{\mathrm{H\!K}}

\newcommand\Hom{\operatorname{Hom}}

\newcommand\res{\operatorname{Res}}

\newcommand\SOs{\operatorname{\mathscr{SO}}}

\newcommand\Ker{\operatorname{Ker}}

\newcommand\Mod{\operatorname{Mod}}

\newcommand\GL{\operatorname{GL}}

\newcommand\SL{\operatorname{SL}}
\newcommand\Orth{\operatorname{O}}
\newcommand\SO{\operatorname{SO}}
\newcommand\so{\operatorname{\mathfrak{so}}}

\newcommand\Un{\operatorname{U}}
\newcommand\spin{\operatorname{Spin}}
\newcommand\cspins{\mathscr{CS}\mathit{pin}}
\newcommand\spins{\mathscr{S}pin}

\newcommand\teich{\mathscr{T}}
\newcommand\Teich{\operatorname{Teich}}

\newtheorem{theorem}{Theorem}[section]

\newtheorem{lemma}[theorem]{Lemma}

\newtheorem{proposition}[theorem]{Proposition}
\newtheorem{corollary}[theorem]{Corollary}

\theoremstyle{definition}

\theoremstyle{remark} 
\newtheorem{remark}[theorem]{Remark}

\newtheorem{question}[theorem]{Question}
\newtheorem{problem}[theorem]{Problem}

\begin{document}
\author{Eduard Looijenga}
\thanks{Supported by the Chinese National Science Foundation}
\address{Yau Mathematical Sciences Center, Tsinghua University, Beijing (China) and Mathematisch Instituut, Universiteit Utrecht (Nederland)}
\email{e.j.n.looijenga@uu.nl}

\title{Teichm\"uller spaces and Torelli theorems for hyperk\"ahler manifolds}
\begin{abstract}
Kreck and Yang Su recently gave counterexamples to a version of the Torelli theorem for hyperk\"ahlerian manifolds
as stated by Verbitsky. We extract the correct statement and give a  short proof of it.
We also revisit a few of its consequences, some of which are given new (shorter) proofs, and ask some questions.
\end{abstract}

\maketitle
\hfill{\emph{To Shing-Tung Yau,  on the occasion of his 70th birth year.}}

\section*{Introduction}
Kreck and Yang Su \cite{kreck-su} recently noticed that the Torelli theorem as stated by Verbitsky in \cite{verbitsky} cannot hold. This led Verbitsky to post an erratum \cite{verbitsky_a} which purports to resolve the issue. Since many subsequent papers have used his theorem, we thought it worthwhile to  offer, what we hope  is,
a complete account, which starts out from the basics. This led us to set up things a little differently than in the primary sources, as this turns out to have the merit of giving shorter, more transparent  proofs. Among this is our definition of the  Teichm\"uller space $\teich$ of \emph{hyperk\"ahlerian} complex manifold structures given up to
isotopy  on a fixed compact manifold $M$ and  its separated quotient ($\teich$ is almost never separated). This should be distinguished from the 
Teichm\"uller space $\teich_{\HK}$ of \emph{hyperk\"ahler} structures, which is always separated and helps to understand the former. 
We found it also worthwhile to introduce the Teichm\"uller space $\teich_\HH$ of Einstein metrics on $M$, as  some properties of interest here are at the end of the  day  properties of  that space. This  also made it relatively straightforward to construct universal families over the Teichm\"uller spaces in question, thereby  recovering  a recent theorem of Markman \cite{markman:family}. Apart from  this, we believe that  what distinguishes this paper from other accounts are:  the more prominent role of the twistor families,  our Proposition \ref{prop:separation} (which is a key to our  definition of the Teichm\"uller spaces), and the half page proof of the Torelli theorem that is essentially Lemma \ref{lemma:main}.

We here treat a twistor deformation as if its base (a projective line) were a Shimura variety (which it certainly is not),  as this not only is helpful in deriving the Torelli theorem, but also  yields a simple way to formulate---and leads to a short way to obtain---a recent result  of Soldatenkov \cite{soldatenkov} (qualified by him as `folklore') and  Green-Kim-Laza-Robles \cite{gklr} on the period map for the full cohomology of a hyperk\"ahlerian manifold. Strictly speaking this last application  is independent  of the Torelli theorem, but we included it, because this merely comes as a bonus after the ground work done here. Let us mention in this context recent work of Bakker-Lehn \cite{bl}, who in a sense take the opposite approach: they obtain a proof of the Torelli theorem which  avoids twistor families altogether (and which allows  mild singularities on the objects that are parametrized). 
\\

We close this introduction with a brief glance backwards along the road traveled so far. Shortly after the Calabi conjecture became Yau's theorem, it was realized  by a number of people that this could be a tool for investigating  the period map for $K3$ surfaces. The first successful application was independently due to  Siu \cite{siu} and the author \cite{looijenga}, who, by making use of connected chains of twistor conics,  proved that the period map for k\"ahlerian $K3$ surfaces is surjective. There were  no other irreducible hyperk\"ahlerian manifolds known at the time, but it was clear that these proofs would extend to that case, once one had some control on the possible K\"ahler classes. For general hyperk\"ahlerian manifolds this was eventually supplied by the work of Huybrechts
\cite{huy:basic}  (which used the Demailly-P\u{a}un  criterion \cite{dp} for the K\"ahler property as an essential  tool).  Verbitsky was probably the first to have a clear strategy for using  twistor conics to prove  injectivity  as well. 
In either case, the earlier use of chains of twistor conics  served as a template for establishing properties of the period map. But the proof of Lemma \ref{lemma:main} now shows that this path is somewhat roundabout in more ways than one, and that this has
prevented us from recognizing the utter simplicity of the situation. Since for $K3$ surfaces  the Demailly-P\u{a}un  criterion amounts to a classical fact, we can, with this bit of additional hindsight (and ensuing change of the year count), 
even more concur with Huybrechts, who  wrote at the end of his 2011 survey   \cite{huy:bourbaki} of Verbitsky's work ``\textit{To conclude, the Global Torelli theorem for $K3$ surfaces could have been proved along the lines presented here some thirty years ago''}.
\\

It is a pleasure to acknowledge correspondence with Matthias Kreck on some of the issues that arise here. 
I thank Benson Farb, Dick Hain and Andrey Soldatenkov for feedback on a preliminary  draft, in particular Andrey for alerting me to \cite{gkt}.
I thank Gopal Prasad for consultations on central extensions.

I  am also grateful to the referee, whose  careful  reading of the submitted version, led me to correct the manuscript at several places.

\section{Hyperk\"ahlerian manifolds and the twistor construction} 

\subsection*{The twistor construction} A  \emph{holomorphically symplectic manifold}  is (in this paper)  a simply-connected compact complex manifold $X$ which admits an   everywhere nondegenerate holomorphic $2$-form. A theorem of Yau asserts that every K\"ahler form on such a manifold 
contains in its cohomology class a unique  \emph{K\"ahler-Einstein}  metric (which here means that the Ricci form of the metric  is zero). This has important consequences for the deformation theory of such $X$. 

Let us first remember  that on a  finite-dimensional real inner product space $V$, an endomorphism $E\in \End(V)$  is  infinitesimally orthogonal if and only if the form  $(v, v')\in V\times V\mapsto \langle v, E(v')\rangle$ is antisymmetric,
and that this identifies such endomorphisms with $\wedge^2 V^*$. So the K\"ahler form,  the real and the  imaginary part  of a symplectic holomorphic $2$-form,  give three infinitesimal orthogonal transformations of the real tangent bundle. The former reproduces the given complex structure (which is always flat), but the vanishing of the Ricci tensor ensures that the other two are  flat as well. The real span of these three transformations is  then closed under the  Lie bracket, yielding  a copy of the Lie algebra of $\SO (3)$ (which is also that of the unit quaternions $\HH^1$). If we also add the identity, then their span is even closed under composition and the resulting algebra is a copy $\HH_X$ of the quaternions. So  $\HH_X=\RR\oplus\HH_X^\pure$, with  $\HH_X^\pure$ being the Lie algebra just mentioned.
Since the holonomy group of the underlying  Riemann manifold will  centralize  $\HH_X$, that group must be contained in a unitary group over the  quaternions. 

The intersection $\HH^1_X\cap \HH_X^\pure$ (a  $2$-sphere) is the set of  square roots of $-1$ in $\HH_X$. It contains the given 
complex structure, but we now  observe that this is one of many, for every element of this $2$-sphere defines a (new) integrable complex structure for which the metric is K\"ahler.  We refer to this family of complex structures as a \emph{twistor deformation} of $X$. It also explains why $X$ is called a \emph{hyperk\"ahler manifold} when it is endowed with a K\"ahler-Einstein metric. If merely a K\"ahler-Einstein metric exists, then we will say that $X$ is \emph{hyperk\"ahlerian}. We say that a hyperk\"ahler  manifold $X$ is \emph{irreducible} if it does not decompose nontrivially as the product of two holomorphically symplectic manifolds; this is known to be equivalent to  $\dim_\CC H^{2,0}(X)=1$ or (by  Berger's classification of holonomy groups) that every flat endomorphism of its tangent bundle is contained in the copy of the quaternions defined above.

The twistor construction is  best understood by starting out with the underlying Riemann manifold with a metric  (that we shall denote by $N$; the metric is denoted by $g$) of which we assume that the  flat  endomorphisms of the tangent  bundle form  a copy $\HH_N\subset\End(TN)$ of the quaternions. The  last property ensures that we are in the irreducible case. The multiplicative group $\HH_N^\times$  has center $\RR^\times$ and its commutator subgroup consists of the unit quaternions $\HH_N^1\subset \HH_N^\times$ (a copy of $\spin (3)$). These two meet of course in the center of $\HH_N^1$, which is
$\mu_2=\{\pm 1\}$. The Lie algebra  of $\HH_N^1$  is $\HH_N^\pure$ and  $S_N:=\HH^1_N\cap \HH_N^\pure$ is the set of square roots of $-1$ in $\HH_N$ and  is a round $2$-sphere.

The group $\HH_N^\times$ of $\HH_N$ acts on the tangent bundle of  $N$. Hence we have a contra-gradient  action of $\HH_N^\times$ 
on the cotangent bundle and therefore on the space  of $C^\infty$-forms.
The flatness ensures that this action commutes with exterior derivation and its adjoint, so that this action preserves the space of harmonic forms.
We identify this space with $H^\pt(N; \RR)$, so that $H^\pt(N; \RR)$ becomes a $\HH_N^\times$-representation.  Note that by these conventions, the subgroup $\RR^\times\subset \HH_N^\times$ defines the opposite grading of $H^\pt(N; \RR)$ in the sense that $t\in\RR^\times\subset\HH_N^\times$  acts on $H^d(N; \RR)$ as multiplication by $t^{-d}$. The action of $u\in \HH_N^\times$ on $H^{4m}(N; \RR)$ is scalar multiplication with
$(u\overline u)^{-2m}$ and the linear map $H^\pt(N; \RR)\otimes_\RR H^\pt(N; \RR)\to H^\pt(N; \RR)$ defined by the cup product  is one 
of $\HH_N^\times$ representations.

Via the above correspondence, any element of $\HH_N^\pure$ determines a $2$-form on $N$. This $2$-form  is harmonic and we thus obtain an embedding of $\HH_N^\pure$  in $H^2(N; \RR)$. We shall denote its image by $P_N$. Since $\HH^\pure_N$ is naturally oriented, so will be $P_N$. We shall see that in some sense, this oriented $3$-dimensional subspace of
$H^2(N; \RR)$ is almost a complete invariant of the metric $g$. 
It is clear that $P_N$ is invariant under the action of $\HH^\times_N$. If we restrict that action to $\HH^1_N$, then $P_N$ is essentially the adjoint  representation. We transport the  norm on $\HH_N^\pure$ to $P_N$ to obtain a positive quadratic form on $P_N$. This positive quadratic form
defines a conic in the projective plane $\PP(\CC\otimes_\RR P_N)$ that we shall denote---for reasons that become clear later---by $\per(P_N)$.

Each $J\in S_N$ defines an integrable complex structure that turns $N$ into a K\"ahler-Einstein manifold $X_J$. The elements $\omega\in \CC\otimes_\RR P_N$ that satisfy  $\omega(Ja,b)=\omega (a, Jb)=\sqrt{-1}\omega(a, b)$ make up a complex line in $\CC\otimes_\RR P_N$.
Indeed, this is just $H^0(X_J, \Omega^2_{X_J})$. Since we have $J^*\omega=-\omega$, and $J^*$ respects the above quadratic form, it follows that the line $H^0(X_J, \Omega^2_{X_J})$ defines a point of $\per(P_N)$. It is an easy exercise to verify that the map
$J\in S_N\mapsto [H^0(X_J, \Omega^2_{X_J})]\in \per(P_N)$ is a diffeomorphism which is even conformal. 

\subsection*{The associated variation of Hodge structure}
We now can state a fundamental theorem of Hitchin-Karlhede-Lindstr\"om-Ro\v{c}ek (Thm.\ 3.3 of \cite{hklr}) in a form that suits our purpose best.
It states that there exists a complex structure on $N\times \per(P_N)$ making it a complex manifold $\Xcal_N$ such that the projection onto $\per(P_N)$ is holomorphic and if $z\in \per(P_N)$ corresponds to $J\in S_N$, then the fiber over $z$ is just $X_J$. The product metric yields in every fiber a 
K\"ahler metric, but  as Hitchin \cite{hitchin} has shown, $\Xcal_N$ does \emph{not} admit a K\"ahler metric.  It is a remarkable fact that the fibers of the other projection onto $N$ define holomorphic sections of $\Xcal_N\to \per(P_N)$ (called by this community \emph{twistor lines}), but with
normal bundle isomorphic to a direct sum of $\half\dim N$ copies of $\Ocal_{\per(P_N)}(1)$. (Its underlying $C^\infty$ vector bundle is indeed trivial:
$\half\dim N$ is even, and $\Ocal_{\PP^1}(1)^2$ is $C^\infty$-isomorphic to $\Ocal_{\PP^1}(1)\oplus\Ocal_{\PP^1}(-1)\cong \Ocal_{\PP^1}^2$.) So such a section cannot appear as a fiber of a holomorphic map~(\footnote{This counterpoint between two familiar voices---one holomorphic, the other differential-geometric---has been exploited in mathematical physics to great effect.}).

For any $J\in S_N$, the centralizer  of $J$ in $\HH_N^\times$ is the intersection of $\HH_N^\times$ with $\RR+\RR J$
and so is naturally identified with $\CC^\times$.  Via this identification, $\zeta\in\CC^\times$ acts on $H^{p,q}(X_J)$ as multiplication with 
$\zeta^{-p}\bar\zeta^{-q}$. So we thus  recover the Hodge decomposition as an eigenspace decomposition.  
If we regard $\HH_N^\times$ as the group of real points of an algebraic group defined over $\RR$, then this copy of $\CC^\times$ should also be thus understood, namely as $\Ss(\RR)$, where $\Ss:=\res_{\CC|\RR}\GG_m$. This is what is called the \emph{Deligne torus}, whose  \emph{raison d'\^etre} is indeed the observation that a finite-dimensional representation of $\Ss(\RR)$ on a real vector space endows (the complexification of) 
that vector space with a Hodge structure. Here is then a way to sum this up:

\begin{proposition}\label{prop:delignetorusrep}
Let $f: \Xcal_N\to \per(P_N)$ be the projection. Then $f$ is holomorphic and $R^\pt f_*\RR_{\Xcal_N}$ is a constant local system which comes with a natural action  of $\HH^\times_N$. The action of $\HH^\times_N$ on $\per(P_N)$ is transitive and the stabilizer
of any $z\in \per(P_N)$ in $\HH^\times_N$ is a Deligne torus whose representation on the stalk  over $z$ defines the Hodge structure on $H^\pt(X_z; \CC)$.
\end{proposition}

So $\per(P_N)$ not only plays here the role of a period space, but also parametrizes the elements of a conjugacy class of homomorphisms $\Ss(\RR)\to \HH^\times$. This is reminiscent of the data that go into the definition of a Shimura variety.

\begin{remark}\label{rem:lie}
Every  $\eta\in P_N\ssm\{0\}$ is the K\"ahler class for some hyperk\"ahler structure on $N$ and so we have a representation $\rho_\eta$ of  the Lie algebra
$\sllie_2(\CC)$ on $H^\pt(N; \CC)$ for which 
$
(\begin{smallmatrix}
0 &0\\
1 &0
\end{smallmatrix})
$
goes to the cup product with $\eta$ and 
$
(\begin{smallmatrix}
1 &0\\
0  &-1
\end{smallmatrix})
$
acts as on $H^k(N; \CC)$ a multiplication by $2m-k$. Verbitsky observed in 1990 \cite{verbitsky1990} that the span of the images of the $\rho_\eta$'s
generate in the endomorphism Lie algebra of $H^\pt(N; \CC)$ 
a graded Lie subalgebra isomorphic to $\so_5(\CC)$. This graded Lie algebra  is of course defined over $\RR$ and as such isomorphic to $\so(4,1)$.
The  part in degree zero is reductive and  has $\HH_N^\pure$ as its the semisimple part and the span of 
$
(\begin{smallmatrix}
1 &0\\
0  &-1
\end{smallmatrix})
$ as its  center.
\end{remark}

We close this section with: 

\begin{lemma}
The group of automorphisms of a holomorphically symplectic manifold  $X$ which fix a given K\"ahler class,  is finite. 
This is in particular so for the group $\aut_0(X)$ of automorphisms that are isotopic to the identity. If $N$ is an Einstein manifold as  above, then its group
of isometries that are isotopic to the identity, $\aut_0(N)$, coincides with the $\aut_0$  of every fiber of the associated twistor deformation.  
\end{lemma}
\begin{proof}
For the first assertion, just note  that the elements of  $\aut_0(X)$ will fix the K\"ahler-Einstein metric associated with this K\"ahler class and since the automorphism group of a Riemann manifold is a compact Lie group, so is  $\aut_0(X)$.  But a one-parameter subgroup of $\aut_0(X)$ determines a nontrivial holomorphic vector field on $X$, whose  contraction with the symplectic form then produces a nontrivial holomorphic $1$-form. On a simply-connected  complex K\"ahler manifold, these do not exist. 
The other assertions are obvious from the preceding discussion.
\end{proof}

\section{Teichm\"uller spaces and period maps} From now on,  we fix  a compact simply-connected manifold  $M$ of dimension $4m$ which admits an irreducible  hyperk\"ahlerian structure. This structure determines an orientation of  $M$ and an oriented  $3$-plane $P_o$ in $H^2(M; \RR)$.
Since $M$ is simply-connected,  $H:=H^2(M; \ZZ)$ is free abelian. 
According  to Bogomolov, Beauville and Fujiki there exists a nondegenerate quadratic  form $q: H\to \ZZ$
such that for some positive rational number $c$, the identity  $q(a)^m=c\int_M a^{2m}$ holds for all $a\in H$ and for which $q_\RR$ is positive on the oriented  $3$-plane in $P_o$. They prove that the form $q_\RR$ has signature $(3,n)$, with  $n:=\! b_2(M)-3$.  The Grassmannian of positive $3$-planes is contractible---it is the symmetric space of $\Orth(q_\RR)$---so that the tautological $3$-plane bundle over it is trivial. Hence the  orientation of $P_o$ orients the whole  bundle, which means that every positive $3$-plane in $H_\RR$ now comes with an orientation.  We shall refer to this as a \emph{spin orientation} on $H$.

We  make both the orientation on $M$  and spin orientation on $H$ as part of our initial data and so we shall only  consider hyperk\"ahlerian structures  that  induce the given orientation (but as Soldatenkov \cite{soldatenkov} has noted, this is in fact automatically the case) and spin orientation (for which the same property might hold---by a theorem of Donaldson this is the case for $K3$ surfaces). This spin orientation determines for every  positive \emph{oriented} $2$-plane $\Pi$  in $H_\RR$,  a \emph{positive cone}: $\Pi^\perp$ has signature $(1, n)$ and so the set of positive vectors in $\Pi^\perp$ make up an antipodal pair of open cones  and the spin orientation singles out one of them.  

We denote by $h_q: H_\CC\times H_\CC\to \CC$ 
the hermitian extension of the symmetric bilinear form associated with $q_\RR$. 

\subsection*{The period manifold} 
A  hyperk\"ahlerian structure on $M$ turns $M$ into  a complex manifold $X$, so that we have 
a Hodge decomposition $H_\CC= H^{2,0}(X)\oplus H^{1,1}(X)\oplus H^{0,2}(X)$ with $H^{2,0}(X)$ of dimension $1$.
Since the cup product  preserves the Hodge structure on $X$, the Hodge type of $q$ will be  $(-2,-2)$. The above characterization
of $q$ then shows  that the Hodge decomposition is orthogonal for $h_q$, with $h_q$  positive on $H^{2,0}(X)\oplus  H^{0,2}(X)$ and of signature $(1,n)$ on $ H^{1,1}(X)$. It also follows that  $H^{2,0}(X)$ is isotropic for $q_\CC$. Since $H^{0,2}(X)=\overline{H^{2,0}(X)}$, the Hodge decomposition is then completely given by the complex line $H^{2,0}(X)$, which, as we just observed, is isotropic for $q_\CC$ and positive for $h_q$. So such Hodge structures are parametrized by an open subset $\per(H_\RR)$ of 
the nonsingular quadric $\check{\per}(H_\RR)$ (of complex dimension $n+1$) in $\PP(H_\CC)$ defined by $q_\CC$, namely the locus
which parametrizes the  lines that are $h_q$-positive.  (The quadric $\check{\per}(H_\RR)$ is homogenous under its $\Orth (q_\CC)$-action and $\per(H_\RR)$ is an open $\Orth(q_\RR)$-orbit in this quadric.)  It is clear that such a period manifold $\per(V)$ is defined for  any real vector space $V$ equipped with a nondegenerate quadratic form of signature $(p,\dim V-p)$ (but we need $p\ge 2$ to make it nonempty).

Note that a point $z\in \per$ determines an oriented positive $2$-plane $\Pi_z$ in $H_\RR$: for the associated Hodge decomposition, the sum $H_z^{2,0}+H_z^{0,2}$ is the complexification of a $2$-plane $\Pi_z$ in $H_\RR$, which is indeed canonically oriented (and hence  determines a positive cone). So the positive cone we have associated to $\Pi_z$ is an open subset of the real part of $H^{1,1}_z$; we refer to it as the \emph{positive cone defined by} $z$.  Conversely, an oriented positive $2$-plane in $H_\RR$ determines a point of $\per$.

\begin{question}\label{question:symmspace}
If $P\subset H_\RR$ is a positive $3$-plane, then $\per(P)=\per (H_\RR)\cap \PP(P_\CC)=\check{\per}(H_\RR)\cap \PP(P_\CC)$ is a conic.
We prefer to call this a \emph{twistor conic} rather than twistor line, since that name had already been taken (in the early literature of the subject a twistor line is a section of a twistor deformation).
A twistor conic is a maximal irreducible compact subspace of $\per$. Its Douady space is identified with the Grassmannian $\Gr_3^+(H_\CC)$ of $h_q$-positive complex $3$-planes in $H_\CC$, where one should note that the  projective plane defined by such a $3$-plane meets $\per$ in a nonsingular conic. This is a bounded symmetric domain for $\Un(h_q)$ whose real part is the  symmetric space $\Gr_3^+(H_\RR)$ of $\Orth(q)$ which parametrizes the twistor conics. Does this space parametrize geometric  structures on $M$ (in a manner that 
for real $3$-planes gives us the structure  of an Einstein metric)? 
\end{question}

\subsection*{The `main proposition'}
The \emph{Teichm\"uller space} $\teich(M)$  of $M$ is for the moment just a  set, namely the set  of hyperk\"ahlerian structures on $M$ given up to $C^\infty$-isotopy.
By assigning to a hyperk\"ahlerian complex structure on $M$ the associated Hodge decomposition on $H$, we obtain the \emph{period map}
\[
\Per: \teich(M)\to \per(M).
\]
The Kodaira-Spencer theory suggests that $\teich(M)$ has the structure of a (perhaps non-separated)  complex manifold  and the local Torelli theorem would then tell us that $\Per$ is a local isomorphism. We will establish this when we have at our disposal Proposition \ref{prop:separation} below. 
Let  us first observe that for a twistor family this gives us the period map discussed earlier. More precisely,  if $M$ is endowed with an Einstein metric and the resulting Riemann manifold is denoted by $N$, so that we then have defined a twistor family
$\Xcal_N\to \per(P_N)$, then:

\begin{lemma}\label{lemma:killing}
The action of the group $\HH^1_N$ of unit quaternions on $H_\RR$ leaves $q_\RR$ invariant.
We have  $\per(P_N)=\per(H_\RR)\cap \PP(P_\CC)$  and
the tautological map $\per(P_N)\to\teich(M)$ composed with $\Per$ is  the inclusion $\per(P_N)\subset \per(H_\RR)$. $\square$
\end{lemma}

The proof is left as an exercise.
Since we have fixed $M$, we shall from now on write $\per$ for $\per(H_\RR)$ and $\teich$ for $\teich(M)$.
\\

Parts of the following proposition 
appear in somewhat different incarnations (at least implicitly) in various places in the  literature (and then with somewhat different proofs), which makes it hard to give it a proper attribution. 
The  archetypical version is certainly the Main Lemma of Burns-Rapoport \cite{burns-rapoport}, which it amplifies and generalizes. We here replace their use of Bishop's analyticity  theorem by a properness theorem of  Fujiki (which was not available at the time). Part (iv) is due to Hassett-Tschinkel (\cite{ht}, Thm. 2.1).

\begin{proposition}\label{prop:separation} 
Let $\pi:\Xcal\to U$ and $\pi':\Xcal'\to U$ be proper holomorphic families of  hyperk\"ahlerian manifolds over the same simply-connected complex manifold $U$. Suppose we are given an isomorphism between the associated variations of Hodge structure in degree two: $\phi: R^2\pi'_*\ZZ_{\Xcal'}\cong R^2\pi_*\ZZ_{\Xcal}$. If for  some $o\in U$, $\phi_o$ is induced by  isomorphism $f_o: X_o\cong X'_o$, then  there exist a proper generically finite morphism  $\hat U \to U$, 
a closed  analytic subspace $\Zcal\subset\Xcal\times_{\hat U}\Xcal'$  flat over ${\hat U}$,  and 
a closed analytic, proper subset $K\subsetneq {\hat U}$, 
such that 
\begin{enumerate}
\item [(i)] if $\hat u\in {\hat U}\ssm K$ lies over $u\in U$, then  $Z_{\hat u}$ is  the graph of an 
isomorphism $f_{\hat u}:X_{u}\cong X'_{u}$ which is isotopic to $f_o$ and induces $\phi_u$; moreover  $f_o$ appears in this manner: for some  
$\hat o\in \hat U\ssm K$ over $o$, we have $f_{\hat o}=f_o$,
\item [(ii)] for every $u\in U$,  $X_{u}$ and $X'_{u}$ are bimeromorphically equivalent,
\item [(iii)] if there exist  $\kappa\in H^0(U,R^2\pi_*\RR)$ and $\kappa'\in H^0(U,R^2\pi'_*\RR)$ which restrict 
to a K\"ahler class in every fiber of $\pi$ resp.\ $\pi'$ and $\phi_o(\kappa'(o))=\kappa (o)$, then we can take ${\hat U}=U$ and 
$\Zcal$ will be the graph of an $U$-isomorphism $\Xcal\cong \Xcal'$,
\item [(iv)] the group $\aut_0(\Xcal/U)$ of automorphisms of $\Xcal/U$ that are fiberwise isotopic to the identity is finite, specializes for  every $u\in U$  to the  group $\aut_0(X_u)$ of automorphisms of $X_u$ isotopic to the identity, and is via $f_o$ naturally identified with $\aut_0(\Xcal'/U)$. Its also acts on the $U$-morphism $\Zcal\to \hat U$ and has the property that the action transitive on the generic fiber of $\hat U/U$.
\end{enumerate}
\end{proposition}

\begin{proof}
Let $D:=D_{\Xcal\times_U\Xcal'/U}$ be the relative  Douady space which parametrizes the  compact analytic subspaces of $\Xcal\times_U\Xcal'$ contained in a fiber  of $\Xcal\times_U\Xcal'/U$. This exists as an analytic space by a theorem of 
Pourcin \cite{pourcin}, and comes with a universal family $\Zcal_D\subset \Xcal\times_U\Xcal'\times_U D$ that is proper and flat over $D$. Let ${\hat U}$ be the irreducible component of $D$  which contains the graph of  $f_o:X_o\cong X'_o$ and put $\Zcal:=\Zcal_{\hat U}$. 
The map $\Xcal\times_U\Xcal'\to U$ is a \emph{weakly K\"ahler} morphism,  which means that there exists a $2$-form on the source whose restriction to every fiber is a 
K\"ahler form. It follows from work of Fujiki  (see  the last paragraph of $\S 1$ of \cite{fujiki-2}) that the projection $r: {\hat U}\to U$ is proper (he actually  assumes that we have in K\"ahler morphism, meaning that the putative $2$-form is closed, but for this conclusion, the weaker assumption will do). In particular, $r(\hat U)$ is a closed subvariety of $U$. We show that it is all of $U$.

The local Torelli theorem implies that there exists a neighborhood $V$ of $o$ in $U$ such that $f_o$ extends to $V$-isomorphism $\Xcal_V\to \Xcal'_V$. The graph of this isomorphism appears in $\Zcal$ and so $V$ lies in the image of $r$.  A closed subvariety of $U$ which contains a nonempty open set equals $U$ and so $r(\hat U)=U$. It also follows that the locus $K$ of ${\hat u}\in {\hat U}$ for which $Z_{\hat u}$ is not   the graph of an isomorphism is a proper closed analytic subset of ${\hat U}$ and that $r(K)$ is a closed analytic  proper subset of $U$.

The proof of (ii)  follows a standard argument \cite{burns-rapoport}: if $\hat u\in {\hat U}$ lies over $u$, then the  algebraic cycle  $Z_{\hat u}$ on  $X_{u}\times X'_{u}$ is of pure complex dimension $2n$ and contains a unique irreducible component with multiplicity one which projects with degree one on both $X_{u}$ and $X'_{u}$. That component therefore establishes a bimeromorphic equivalence between the two factors. 

In the situation of (iii), each fiber of either family comes with a 
the K\"ahler class. So if we  give each fiber  the associated Einstein metric, then   
$Z_{\hat u}$ will be  the graph of an isometry whenever it is the graph of an isomorphism. But the fiber  metric depends continuously on the base point
and so even when $\hat u\in K$, the correspondence $Z_{\hat u}$ will then implement an isometry of an open-dense subsets of $X_u$ onto one of  $X'_u$. As  it will take a Cauchy sequence to a Cauchy sequence, this implies that $Z_{\hat u}$ is in fact the graph of an isometry and hence of an isomorphism.  In other words, $K=\emptyset$. So  ${\hat U}$ now parametrizes isometries in the same isotopy class.  The local 
Torelli theorem then implies that $\hat U\to U$ is an unramified covering. But as $U$ is simply connected and $\hat U$ irreducible, $\hat U\to U$ must be an isomorphism.

We now prove (iv). Let $u\in U$. Since $\aut_0(X_u)$ acts as the identity on $H_2(X_u; \RR)$, it  fixes a K\"ahler class. This class is uniquely represented by a  K\"ahler-Einstein metric, which is then also preserved by the finite group $\aut_0(X_u)$. 
Let $f_u\in \aut_0(X_u)$, and assume it has  finite order, $d$ say.  If we apply part (iii) to two copies of $\Xcal/U$ with $\phi$ the identity and $(o, f_o)$ replaced by $(u, f_u)$, then we find that $f_u$ extends to an automorphism $F$ of $\Xcal/U$ which will have that same order $d$ in every fiber. 
This also applies to $\Xcal'/U$, and as their restrictions over ${\hat U}\ssm K$ are the same, (iv) follows.

It is clear that the group $\aut_0(\Xcal/ U)$ acts (by precomposition) on the $U$-morphism $\Zcal\to \hat U$.
Let ${\hat u}_1, {\hat u}_2\in {\hat U}\ssm K$ lie  over the same point $u\in U$. We claim that  they then lie in the same $\aut_0(\Xcal/ U)$-orbit.
Our  assumption means that $f_{{\hat u}_1}$ and $f_{{\hat u}_2}$ differ by  an automorphism $f_u$ of $X_{u}$.  Since ${\hat U}\ssm K$ is connected, $f_u$ must be isotopic to the identity. By (iv), $f_u$ is then the specialization of an $F\in \aut_0(\Xcal/{U})$. So precomposition with $F$  defines  an automorphism of ${\hat U}/U$ which takes ${\hat u}_1$ to ${\hat u}_2$. 
In particular, $\aut_0(\Xcal/ U)$ acts transitively on the generic fiber of $\hat  U/U$.
\end{proof}

We do not know whether the morphism $\hat U\to U$ appearing in this proposition is always an isomorphism.

\begin{remark}\label{rem:bimer}
As is well-known, any bimeromorphic equivalence $f: X\dashrightarrow X'$ between compact complex manifolds is a morphism on the complement of a closed analytic subset $Y\subset X$ of complex codimension $\ge 2$ and hence takes a holomorphic $p$-form on $X'$ to one on $X$.
Assuming now that both $X$ and $X'$ are hyperk\"ahlerian, then for a nondegenerate holomorphic $2$-form $\alpha'$ on $X'$, $f^*\alpha'$ will be a nondegenerate $2$-form on $X$, so that  $f|X\ssm Y$ must be a local isomorphism (in other words, there is no contraction in codimension one). The same argument applied to $f^{-1}$  proves that we can arrange that there exist closed analytic subsets $Y\subset X$ and $Y'\subset X'$  of complex codimension $\ge 2$ such that
$f$ maps $X\ssm Y$ isomorphically onto $X'\ssm Y'$. In particular, $f$ induces an isomorphism $H^2(X'; \ZZ)\to H^2(X; \ZZ)$ of Hodge structures.
A theorem of Huybrechts (Thm.\ 2.5 in \cite{huy:kcone}) asserts that $f$ can arise as  the specialization for a situation  as in Proposition  \ref{prop:separation}, with $U$ the complex unit disk and $\Zcal$ being
over $U\ssm\{ 0\}$ the graph of an $(U\ssm\{ 0\})$-isomorphism. This implies  that $f$ determines an isotopy class $[f]$ of diffeomorphisms between the underlying manifolds of $X$ and $X'$ which induces on $H^2$ the above map: it makes $f$ a limit of graphs of diffeomorphisms of $X$ onto $X'$.
\end{remark}

\begin{question}\label{question:sep}
It is possible to  reconstruct $[f]$ from $f$ without recourse to one-parameter deformations as above? For example, if the underlying manifold of both $X$ and $X'$ is $M$, how can we tell whether $[f]$ is in the isotopy class of the identity?
\end{question}

\subsection*{Teichm\"uller spaces} 
We are now ready to give $\teich$ the structure of a possibly nonseparated complex manifold (and hence with a topology) by endowing $\teich$ with an atlas whose charts are of the following type. Given an open subset  $U$ of $\per$,  then let us agree that a \emph{basic chart for $\teich$ 
with domain} $U$  is given by a complex structure
on $M \times  U$ for which the resulting complex manifold  $\Xcal$ has the property that 
\begin{enumerate}
\item[(i)]  the projection $\Xcal\to U$ is holomorphic,
\item[(ii)] the fibers of  $\Xcal\to U$  are hyperk\"ahlerian  manifolds and
\item[(iii)] its period map is given by the inclusion of $U$ in $\per$.
\end{enumerate}
It is clear that such an object defines an injection of $U$  in $\teich$.  By the local Torelli theorem, every hyperk\"ahlerian complex structure on $M$  appears as a  member of such a family. In other words, the basic charts cover all of $\teich$. We give $\teich$ the quotient topology, that is, the finest
topology, for which all the basic charts are continuous. It follows from Proposition \ref{prop:separation} (with $\phi$ the identity and $f_o$ isotopic to the identity) that the locus where two basic charts  with domains $U$ and $U'$ of $\per$  agree,   is the complement of a \emph{closed} (analytic) subset of $U\cap U'$.  This implies that each basic chart is an open map. It is now obvious that our atlas is complex-analytic and that it gives 
$\teich$  the structure of a (non-separated) complex manifold for which $\Per$ is a local isomorphism.

This also suggests that we define \emph{separated Teichm\"uller space} 
$\teich_s$ as follows:  identify two members of our atlas with the same domain  if the hypotheses of Proposition \ref{prop:separation} are satisfied with $\phi$ the identity and $f_o$ isotopic to the identity. In other words, two hyperk\"ahlerian complex structures on $M$ which give complex
manifolds $X$ and $X'$, define the same point of $\teich_s$ if and only if there exist basic  charts $\Xcal/U$,  $\Xcal'/U$ containing $X$ resp.\   $X'$ over the same open subset $U\subset \per$, and a  sequence $(z_i\in U)_{i=1}^\infty$ converging to some $o\in U$ such that $X_{z_i}$ and $X'_{z_i}$ differ  by a $C^\infty$-isotopy and $X_o=X$ and $X'_o=X'$, or equivalently, if there exists a bimeromorphic equivalence $f:X\dashrightarrow X'$ whose associated isotopy class $[f]$ is  that of the identity of $M$ (this motivated our Question \ref{question:sep}).
The space $\teich_s$ is indeed a separated  complex manifold and  the period map factors through the \emph{separated period map}
\[
\Per_s: \teich_s\to \per,
\] 
which of course is still a local isomorphism. 

\begin{remark}[Comparison with Verbitsky's Teichm\"uller space]\label{rem:comparison}
Verbitsky defines in \cite{verbitsky} his  Teichm\"uller space $\Teich$ as the orbit space of the space of hyperk\"ahlerian structures on $M$ with respect to the action of the group of diffeomorphisms isotopic to the identity. Our atlas consists of certain sections to orbits, and so we have \emph{a priori} a map
$\teich\to\Teich$ for which $\Teich$ has the quotient topology.  Since this map is a bijection, it must be  homeomorphism. Def.\ 1.13 of \emph{op.\ cit.}\ introduces a notion of inseparability: two points of a topological space are said to be \emph{inseparable} if every neighborhood of one meets every neighborhood of the other. As Verbitsky observes, this  relation fails in general to be transitive, and  indeed, much work in \cite{verbitsky}  goes into proving  that this is an equivalence relation for $\Teich$ (as pointed out by Matthias Kreck, the proof is distributed all over the paper and seems entangled with the proof of the Torelli theorem). It is however clear from the preceding that for $\teich$ this is just the relation that says that the two points lie in the same fiber of  $\teich\to\teich_s$. Hence it is an equivalence relation, so that our $\teich\to\teich_s$ can be identified with Verbitsky's $\Teich\to\Teich_b$. 
\end{remark}

\subsection*{Other moduli spaces} There is a good reason to consider also two related Teichm\"uller spaces, if only to better understand the formation of the separated quotient above. One is the space $\teich_{\HK}$  of hyperk\"ahler   structures on $M$ given up to $C^\infty$-isotopy and with the metric given up to a scalar (or normalized such that $M$ has unit volume). In view
of the discussion above this amounts to specifying in addition a ray in $H_\RR$ (or rather, in  $H^{1,1}(X; \RR)$) spanned by a K\"ahler class. 
In particular, if $\per_{\HK}$ denotes the space of pairs $(z,r)$ with $z\in \Per$ and $r$ a ray in the positive cone defined by $z$, then in an evident manner we have defined a \emph{hyperk\"ahler period map} $\Per_{\HK}: \teich_{\HK}\to \per_{\HK}$.
Note that the projection $\per_{\HK}\to \per$ is a locally trivial fiber bundle with fibers having the structure of a hyperbolic $n$-space.

\begin{corollary}\label{cor:artincover}
The moduli space  $\teich_{\HK}$ is a separated manifold of dimension $3n+2$ such that $\Per_{\HK}$ is a local diffeomorphism.
The natural map $\pi_{\HK}:\teich_{\HK}\to \teich$ is open, with each fiber having  the structure of a convex open set in an $n$-dimensional hyperbolic space.
 \end{corollary} 
 \begin{proof}
The first assertion follows from Property (iii) of  Proposition  \ref{prop:separation}.  The openness and convexity properties are general facts, which hark back  to Kodaira.  In our case, the rays  in  the positive cone of $X$ make up a real hyperbolic space of dimension $n$, and so the space of rays spanned by 
K\"ahler classes make up an open convex subset this space.
\end{proof}

So the composite $\teich_{\HK}\to\teich\to \teich_s$ 
is a submersion of separable manifolds. Its fibers are disjoint unions of convex open sets in a hyperbolic  $n$-space  and the factorization can be understood as a  topological Stein factorization. Perhaps $\teich$ is best understood via the following characterization.

\begin{corollary}\label{cor:sectionchar}
A section of $\teich \to \teich_s$ over an open subset $U\subset \teich_s$ is given by a continuous section of 
$\teich_{\HK}\to\teich_s$ given up to homotopy.
\end{corollary}
\begin{proof}
This is merely the observation that each homotopy class of sections over $U$ has a natural convex structure, hence is canonically contractible. \end{proof}

If we only retain (apart our initial data) the Einstein metric  (and so do not single out a complex structure for which the metric is K\"ahler), then we obtain another 
Teichm\"uller space  $\teich_{\HH}$~(\footnote{This subscript is intended to honor Hamilton and does not stand for anything \emph{hyper}.})  of 
Einstein metrics on $M$ for which $M$ has unit volume, again given up to isotopy.
We have a natural projection $\teich_{\HK}\to \teich_{\HH}$ and we give $\teich_{\HH}$ the quotient topology. The twistor construction makes it clear that the evident projection  $\teich_{\HK}\to \teich_{\HH}$ is a locally trivial $S^2$-bundle and  that the  `period map'
\[
\Per_{\HH}: \teich_{\HH}\to \Gr_3^+(H_\RR),
\]
which assigns to an Einstein metric $g$ on $M$ the subspace $P_{(M,g)}$, is a 
local diffeomorphism. Note that its target $\Gr_3^+(H_\RR)$ is the symmetric space of $\Orth (q_\RR)$, so that
the arithmetic  group $\Orth(q)$ acts properly discretely on it.

Summing up, we have a commutative diagram 
\begin{equation}\tag{$\dagger$}
\begin{CD}
\teich_\HH@<<< \teich_\HK@>>> \teich\\
@V{\Per_\HH}VV  @V{\Per_{\HK}}VV @V{\Per}VV\\
\Gr^+(H_\RR)@<<< \per_\HK @>>> \per
\end{CD}
\end{equation}
in which the horizontal maps are forgetful and the vertical ones are period maps. All spaces are manifolds, and all but $\teich$ are Hausdorff.
The left square is cartesian (the left pointing maps are 2-sphere bundles) and the right pointing maps have contractible fibers.

\section{A Torelli type theorem}
In this paper, we stipulate that the \emph{mapping class group} $\Mod (M)$ of $M$ is the connected component group of the group of diffeomorphisms of $M$ 
which preserve the initial data, i.e.,  the orientation on $M$ and the spin orientation on $H^2(M; \RR)$. It is clear that this group acts naturally  on all the Teichm\"uller spaces that we introduced. 

Let $\rho$ be the representation of $\Mod (M)$ on $H^2(M)$ and denote 
by $\G_M\subset \GL(H)$ its image.  It is clear that $\Gamma_M\subset \Orth(q)$, but with our definition of $\Mod (M)$  we land in fact in an index $2$ subgroup of $\Orth(q)$, namely the kernel $\Orth^{\#}(-q)$  of the spinor norm for $-q$. As Verbitsky had noticed, a theorem of Sullivan implies  that  $\Gamma_M$ is of finite index in $\Orth(q)$.

\begin{theorem}[A Torelli theorem for hyperk\"ahlerian manifolds]\label{thm:main}
The period map $\Per_s: \teich_s\to \per$ maps every connected component of $\teich_s$ isomorphically onto $\per$.
In particular, the $\Mod(M)$-stabilizer of a component acts with finite kernel on $H^2(M; \ZZ)$.
\end{theorem}

\begin{remark}\label{rem:ky}
Verbitsky \cite{verbitsky} claimed in addition that  $\Per_s$ is a finite covering, but as Kreck and Yang Su \cite{kreck-su} have shown, this is not always the case. 
They invoke a general theorem of Sullivan, which implies  that the kernel of the representation of  $\Mod (M)$ on the full cohomology $H^\pt (M; \ZZ)$ has a finitely 
generated torsion free unipotent group as a subgroup of finite index (provided that $m\ge 2$), and  show that this unipotent group is nontrivial when $M$ is 
the manifold underlying  a Kummer 4-fold  (a projective Kummer $2m$-fold is obtained by taking  for an abelian surface the Hilbert scheme of its length $(m+1)$-subschemes  
of which  the image under the sum map is the origin). So $\teich$ has then infinitely many connected components and there exist elements in $\Ker(\rho)$ 
of which no nontrivial power can appear in the monodromy group of a connected (holomorphic) family of hyperk\"ahler manifolds. 
\end{remark}

Theorem \ref{thm:main} can be considered as a global Torelli theorem for a single component of the Teichm\"uller space of $M$. 
A (weak) version of a global  Torelli theorem for the full Teichm\"uller space is then obtained by combining it 
with a finiteness result of  Huybrechts \cite{huy:finite}, which implies that $\Ker(\rho)$ acts properly on 
the connected component set $\pi_0\teich$ of $\teich$  and has in $\pi_0\teich$ only finitely many orbits. 
The  orbit space $\Ms:=\Ker(\rho)\backslash \teich$ may be considered as a moduli space of  \emph{marked} hyperk\"ahlerian manifolds whose
separated quotient $\Ms_s$ is just $\Ker(\rho)\backslash \teich_s$. The period map clearly factors through  a morphism
 $\Ms_s\to\per$ and the finiteness property and the Torelli theorem above imply  that this is a \emph{finite} (trivial) covering map. By construction, this covering map comes with a faithful action  of $\Gamma_M$. It also follows  that the image under $\rho$ of the $\Mod(M)$-stabilizer of a connected component of $\teich$ is of finite index in $\Gamma_M$ and is hence of finite index in $\Orth(q)$.  We thus find:

\begin{corollary}[A weak global Torelli theorem]\label{cor:weaktorelli}
The set of hyperk\"ahlerian complex structures on $M$  with a prescribed Hodge structure on $H^2(M; \ZZ)$ is nonempty  and decomposes into a finite number of complete bimeromorphic equivalence classes. $\square$
\end{corollary}

\begin{problem}\label{problem}
Find an invariant for Einstein metrics on $M$ that is locally constant under deformation and which separates the connected components of $\teich_{\HH}$, at least up to finite ambiguity. 
\end{problem} 

\begin{remark}
It often happens that the action of $\Mod(M)$ on $\pi_0(\teich)$ is transitive so that $\Ms$ is connected. This is for example so for $K3$ surfaces: by a famous result of 
Piatetski-Shapiro and  Shafarevich, a $K3$ surface which contains $16$ pairwise  disjoint smooth rational curves is a (resolved) Kummer surface  and such surfaces are dense in $\teich(K3)$. On the other hand these resolved Kummer surfaces can be parametrized by the (connected) space
of  complex structures on the $4$-torus $\RR^4/\ZZ^4$, i.e., $\SL_4(\RR)/\Un_2$. It  seems unknown whether $\teich(K3)$ itself is connected, or what amounts to the same, whether $\Mod(K3)$ acts  faithfully on the second cohomology. (This in contrast to its  topological counterpart: for every closed simply connected $4$-manifold $S$, a self-homeomorphism of $S$ which acts trivially on the homology is by a theorem of Kreck \cite{kreck} pseudo-isotopic to the identity and then
a theorem of Perron \cite{perron} implies that such a homeomorphism  is isotopic to the identity through homeomorphisms.)

A similar argument might apply to the case when $M$ underlies the Hilbert scheme of subschemes of a $K3$ surface of length $m$ or underlies a Kummer 
$(2m)$-fold, so that $\Ms$ is then also connected. This is perhaps the reason that some authors prefer to work with moduli spaces of \emph{marked} 
hyperk\"ahlerian manifolds  instead of Teichm\"uller spaces. This amounts to starting with an abstract lattice $\Lambda$ endowed with a nondegenerate 
quadratic form $q_\Lambda$ that is isometric to $H$ and to consider hyperk\"ahlerian manifolds $X$ endowed with an isomorphism of lattices $H^2(X; \ZZ)\cong \Lambda$. 
But this approach can be too coarse in that it ignores too much of the underlying topology.
For example, such a moduli space has at least as many components as the  index of  $\Gamma_M$ in $\Orth(q_\Lambda)$ and this is at least two (due to the spinor orientation).
But even if we specify such an orientation for $q_\Lambda$ and demand that the marking respects spinor orientation, then $\Gamma_M$ need not map onto $\Orth^{\#}(-q_\Lambda)$:  there might exist  two hyperk\"ahlerian structures on $M$, such that an isometry of  their Hodge structures is not induced by a 
bimeromorphic equivalence. Such phenomena  are sometimes eliminated by insisting that the Hodge isometry preserves a bit more than what meets 
the eye,  the `more' being of a topological nature. It  would then be  somewhat misleading  to claim that therefore the Torelli theorem fails. 
For example, Markman proved in \cite{markman:monodromy}  that if $M$ underlies the Hilbert scheme of length $m$ subschemes of a $K3$ surface,  
then  $\Gamma_M$  is a proper subgroup of $\Orth^{\#}(-q)$ unless $m-1$ is a prime power and so then  
such additional components will come up. But as this is merely a reflection  of a topological property of $M$, namely that $\Mod (M)$ must preserve 
certain invariants that are defined on a finite quotient of $H^2(M; \ZZ)$, these connected components are in a sense extraneous---or put differently, some markings are better than others.
\end{remark}

Since $\per$ is simply connected, Theorem \ref{thm:main} is equivalent to saying that $\Per_s$ is a covering map. This is in fact what we will prove and  indeed, it is  implied by: 

\begin{lemma}\label{lemma:main}
Let  $U\subset \per$ be a  neighborhood of $\Per(t)$ isomorphic to the open unit ball in $\CC^{n+1}$. If $t\in \teich_s$  lies over the center of this ball, 
then there exists a unique section $\sigma$ of $\Per$ over $U$ which takes $\Per_s(t)$ to $t$.
\end{lemma}

The proof of this lemma  involves little  more than twistor deformations and  the following theorem of Huybrechts (that is  based on work of Demailly-P\u{a}un \cite{dp}) which ensures that there are enough of these.

\begin{proposition}[Huybrechts \cite{huy:basic}]\label{prop:cone}
Let $X$ be a hyperk\"ahlerian manifold  for which $H^{1,1}(X)\cap H^2(X; \ZZ)=\{0\}$ (in other words, if $\Pi_{z}^\perp\cap H=\{ 0\}$). Then every element of the positive cone of $X$ represents a K\"ahler class. $\square$
\end{proposition}

Let $V$ be a real vector space defined over $\QQ$. For  a linear subspace $W$ of $V$, we define its \emph{rational closure}  to be the smallest linear subspace of $V$ defined over $\QQ$ which contains $W$. If this is all of $V$, then we
say that $W$ is \emph {fully irrational}. It is clear that in the Grassmannian of all linear subspaces of $V$ that are not fully irrational form a countable union of proper subvarieties defined over $\QQ$. In particular, the fully irrational subspaces are dense. 

Proposition \ref{prop:cone} and Lemma \ref{lemma:killing} imply:

\begin{corollary}\label{cor:twistorline}
Let $P$ be a fully irrational, positive $3$-plane in $H_\RR$. Then $\Per_s$ maps every connected component of $\Per_s^{-1}\per(P)$ isomorphically onto $\per(P)$.
\end{corollary}

The proof of Lemma \ref{lemma:main} below re-arranges and simplifies the corresponding treatment in \cite{verbitsky} and \cite{huy:bourbaki}.

\begin{proof}[Proof of Lemma \ref{lemma:main}]
We  can identify $U$ with the open unit ball $B_{<1}$  in $\CC^{n+1}$.
Let $r$ be the supremum of the $a\in (0,1]$ for which there exists a section over the open ball $B_{<a}$ defined by $\rho<a$ that takes $\Per_s(t)$ to $t$.
Then $r>0$ and we must show that $r=1$. Suppose $r<1$. Since $\Per_s$ is a local homeomorphism between separated spaces, two sections defined on the same  connected subset of $\per$ are equal when they are equal at some point.
So if $B_r$ denotes the ball $\rho\le r$, then we have a section $\sigma$ defined over  its interior $B_{<r}$. 
Let $z\in \partial B_r$. A positive line $\ell$ in $\Pi_z^\perp$ determines a twistor conic $\per(\ell+\Pi_z)$. We can (and will) take $\ell$ such that
$\per(\ell+\Pi_z)$ is transversal to the tangent space of $\partial B_r$ at $z$ (an open, nonempty condition) and $\ell$  is fully irrational
(this condition is dense). Then $\ell+\Pi_z$ is fully irrational and $B_r\cap \per(\ell+\Pi_z)$ is near $z$ a manifold with boundary, with $z$ being a boundary point. It follows from Corollary  \ref{cor:twistorline} that $B_{<r}\cap \per(\ell+\Pi_z)$ is nonempty and that the restriction of $\sigma$ to this subset   extends across $z$. So we have a section $\sigma_z$ on an open ball neighborhood $U_z$  of $z$  in $U$ such that $\sigma$ and $\sigma_z$ take the same value in
some point of $U_z\cap B_{<r}$. Since $U_z\cap B_{<r}$ is connected, it follows that $\sigma$ and $\sigma_z$ coincide on $U_z\cap B_{<r}$.
A useful feature of taking the $U_z$ to be open balls is that if $U_z$ and $U'_z$ meet (with $z,z'\in \partial B_r$), then both $U_z\cap U'_z$ and $U_z\cap U'_z\cap B_{<r}$ are connected. For it then follows  that $\sigma$ and the collection $\{\sigma_z\}_{z\in \partial B_r}$ together define a section of $\Per_s$ on  a neighborhood of $B_r$. Since such a neighborhood contains an open ball of radius $>r$, we get a contradiction.
\end{proof}

\section{Refinements and other consequences}

\subsection*{The other two period maps}
Let $\Ccal$ be a connected component of $\teich_\HH$.
Its preimage $\Ccal_{\HK}$  under the projection 
$\teich_{\HK}\to \teich_\HH$ is then a connected component of $\teich_\HK$ and  the image of $\Ccal_{\HK}$  in $\teich$
is a connected component  of $\teich$. Since the Torelli theorem asserts that the period map identifies the separated quotient of the latter with $\per$,
it follows that the other period maps define open embeddings $\Ccal_{\HK}\hookrightarrow \per_{\HK}$ and 
$\Ccal\hookrightarrow \per_\HH$. Our goal is to say something about their  images and to prove that these components come with universal families.

We begin with stating a stronger form of Proposition \ref{prop:cone}. Let us say that a linear form $\delta: H\to \ZZ$ is \emph{negative} if its kernel has signature $(3, n-1)$. 
This of course equivalent to the quadratic form  $q^\vee:H^\vee\to \QQ$ that is inverse to $q$, taking a  negative value on $\delta$. We shall write $H_\delta$ for the kernel
of $\delta: H\to \ZZ$.
We need the following the two theorems.

\begin{proposition}[Boucksom, \cite{boucksom}, Thme.\ 1.2]\label{prop:kcone}
Let $X$ be  hyperk\"ahlerian manifold. Then the set of K\"ahler classes in the positive cone in $H^{1,1}(X; \RR)$ is the intersection of the  positive cone with the open half spaces defined by the fundamental classes of the irreducible rational curves $C$ on $X$ that are negative (defined by $\int_C > 0$). $\square$
\end{proposition}

We also need the following result, due to Amerik-Verbitsky (see also Mongardi \cite{mongardi}, Thm.\ 1.3).

\begin{proposition}[Amerik-Verbitsky \cite{amverb2015}, Thm.\ 5.15 (i) $\Leftrightarrow$ (v)]\label{prop:walls}
For an Einstein metric $g$ on $M$, the set of negative vectors in $H^\vee$ whose kernel does \emph{not} contain the associated $3$-plane $P_g$ only depends on the connected component of the image of $(M,g)$ in $\teich_\HH$. $\square$
\end{proposition} 

This suggests that we define $\Delta_\Ccal\subset H^\vee$ as the set of indivisible  negative  linear forms $H\to \ZZ$, which 
for some K\"ahler-Einstein metric associated with an element of $\Ccal$, are representable by an irreducible  rational curve.

We now need  some elementary properties of arrangements on the Grassmannian $\Gr^+_3(H_\RR)$.
If $\delta\in H^\vee$ is negative, then $\Gr^+_3(H_{\delta, \RR})$ is a codimension $3$-submanifold of $\Gr^+_3(H_\RR)$. In particular,  $\Gr^+_3(H_\RR)\ssm \Gr^+_3(H_{\delta, \RR})$ is simply connected.

\begin{lemma}\label{lemma:finitenesscriterion}
Let $\Delta$ be a subset of $H^\vee$ which consists of indivisible  \emph{negative}  vectors, and  is invariant under a subgroup $\Gamma$ of $\Orth(q)$  of finite index. Then
\[
\Gr^+_3(H_\RR)_\Delta:=\Gr^+_3(H_\RR)\ssm \cup_{\delta\in \Delta} \Gr^+_3(H_{\delta, \RR})
\]
 is open in $\Gr^+_3(H_\RR)$ if and only if  $\Gamma$ has finitely many orbits in $\Delta$. If these equivalent conditions are fulfilled, then 
$\{\Gr^+_3(H_{\delta, \RR})\}_{\delta\in\Delta}$  is locally finite on  $\Gr^+_3(H_\RR)$ so that any open subset of $\Gr^+_3(H_\RR)$ which contains
$\Gr^+_3(H_\RR)_\Delta$ is simply connected, but if they are not, then $\Gr^+_3(H_\RR)_\Delta$ has empty interior.
\end{lemma}
\begin{proof}
If $P$ is a positive $3$-plane, then it is clear that for every $0<\varepsilon<1$, 
the open subset $U_\varepsilon(P^\perp)$ of $H_\RR$ consisting of $(v,w)\in P\oplus P^\perp=H_\RR$ with $q(v)< -\varepsilon q(w)$ is an open 
neighborhood of $P^\perp\ssm \{ 0\}$ which consists of negative vectors.  The lemma follows from the assertion that such an open subset 
meets every $\Gamma$-orbit in $H$ in a finite set, and meets every infinite  union of $\Gamma$-orbits consisting of negative indivisible vectors 
in an infinite set. The proof of this last property is left to the reader.
\end{proof}

The following Corollary has (apart from its last assertion)  a version for $\teich_\HK$ and is in that form due to Amerik-Verbitsky (\cite{av}, Theorem 4.9).

\begin{corollary}
Let $\Ccal$ be a connected component of $\teich_\HH$. Then  $\Delta_\Ccal$ is a finite union of $\Mod(M)_\Ccal$-orbits and  $\Per_\HH$ maps $\Ccal$ diffeomorphically onto $ \Gr^+(H_\RR)_{\Delta_\Ccal}$.  
In particular, $\Ccal$ is simply connected.
\end{corollary} 
\begin{proof}
The Torelli theorem \ref{thm:main} implies that $\Per_\HH$ defines an open embedding  of $\Ccal$ in $\Gr^+(H_\RR)$.
This image is of course $\Mod(M)_\Ccal$-invariant.
Propositions  \ref{prop:kcone}  and \ref{prop:walls} imply that this image in $\Per_\HH$ must be contained in $\Gr^+(H_\RR)_{\Delta_\Ccal}$.
Lemma \ref{lemma:finitenesscriterion} then tells us that $\Delta_\Ccal$ must be a finite union of $\Mod(M)_\Ccal$-orbits and defines
a locally finite arrangement  on $\Gr^+(H_\RR)$. Then turning back to Proposition \ref{prop:kcone}, we see that this  implies that the image of $\Ccal$ is exactly $\Gr^+(H_\RR)_{\Delta_\Ccal}$.
\end{proof}

The Torelli theorem asserts among other things that the $\Mod(M)$-stabilizer of $\Ccal$, $\Mod(M)_\Ccal$, acts with finite kernel on $H=H^2(M; \ZZ)$. 
Property (iv) of Proposition \ref{prop:separation} implies that for every Einstein metric $g$ on $M$ which represents a point of 
$\Ccal$, the group of isometries of $(M,g)$ that are isotopic to the identity only depends on $\Ccal$ and hence can be
identified with the kernel of this action. We therefore denote this kernel by $\aut_0(\Ccal)$.

\begin{corollary}\label{cor:fine_einstein}
The Teichm\"uller space of Einstein metrics on $M$, $\teich_{\HH}$, carries a family of Einstein manifolds $\Ncal_{\HH}/\teich_{\HH}$ which is endowed with a faithful action of $\Mod(M)$. It is  almost-universal in the sense that every family of  Einstein metrics on $M$ is a pull-back of this one, but 
can be so in more ways than one, with the ambiguity residing  in a finite group which is constant on every connected component  $\teich_{\HH}$. 
\end{corollary} 

In somewhat fancier language: $\teich_{\HH}$ underlies a (Deligne-Mumford) stack and this stack is a  constant gerbe on every connected component. 

Let us note that the  fibers of $\teich_\HK\to \teich_\HH$ are twistor conics, and hence  come with tautological families. On the other hand the fibers
of $\teich_\HH\to \teich$ are convex subsets of hyperbolic $n$-space (K\"ahler cones) and these too, come with tautological families.
This is why this  corollary is essentially equivalent to a recent theorem of Markman \cite{markman:family}
which states the same property for $\teich$.

\begin{corollary}\label{cor:family}
The Teichm\"uller spaces $\teich_{\HK}$ and $\teich$ carry families of hyperk\"ahler resp.\ hyperk\"ahlerian manifolds that are almost-universal
in the  sense above. Each connected component $\Ccal$ resp.\ $\Ccal_{\HK}$ comes with a faithful action of an extension of $\Mod(M)_\Ccal$ by
$\aut_0(\Ccal)$. 
\end{corollary}

\begin{proof}
For $\teich_{\HK}$ this is immediate from  Corollary \ref{cor:fine_einstein}. The corresponding result for $\teich$ then follows from
the fact that we have a descent along $\teich_{\HK}\to \teich$.
\end{proof}

As it suffices to prove Corollary \ref{cor:fine_einstein} per connected component of $\teich_{\HH}$, we check this for $\Ccal$. Before we get into the proof, we make some preliminary observations, which we hope help to clarify that the issue is the possible non-triviality of the center of $\aut_0(\Ccal)$. 

Since we need to glue the 
local universal deformations into a global object over $\Ccal$,  the group $\aut_0(\Ccal)$ prevents us, at least \emph{a priori},  to do this in a canonical fashion.
To make this more concrete: suppose that we have two maps $f_+,f_-$ from the $2$-disk $D^2$ to $\teich_\HH$ that coincide on the boundary, 
so that together they  define a map $f:S^2\to \teich_\HH$. The goal is to produce a family of Einstein manifolds over $S^2$ for which $f$ is the classifying map.
Since $D^2$ is contractible, there is no difficulty in finding two corresponding families of Einstein metrics 
$\{g^\pm_x\}_{x\in D^2}$ on $M$. By definition the two metrics that we get when $x\in\p D^2$ differ by an isotopy. We fix some $p\in \p D^2$. 
By modifying the family defined by $f_+$ with the isotopy that we get for $p$, we can already assume that the two metrics coincide in $p$: $g^+_p=g^-_p$. 
If  $\gamma: [0, 2\pi]\to \p D^2$  parametrizes $\p D^2$, with $\gamma(0)=\gamma(2\pi)=p$, then there is a unique continuous family of  isotopies 
$t\in [0, 2\pi]\mapsto h_t$ of $M$ with $h_0$ the identity, such that $h_t$ takes $g^-_{\gamma(t)}$ to $g^+_{(\gamma(t)}$. 
Then $h_{2\pi}$ will be an $\aut_0(\Ccal)$-equivariant isometry of $(M,g^+_p)$ and hence is given by a central element of $\aut_0(\Ccal)$.  It is not hard to check that this central element only depends on the homotopy class of  $f$.

In order that we have a family over $S^2$, we want $h_{2\pi}$ to be  the identity.  But the component $\Ccal$ is obtained from the manifold  $\Gr^+(H_\RR)$ by removal of an (infinite) arrangement of closed submanifolds of codimension $3$,  with the link of each member representing  a nontrivial element of its second homotopy group. So it may well happen that $f$ is not null-homotopic. In fact, since $\Gr^+(H_\RR)$ is contractible,  such a link is  the basic case to consider. The following lemma addresses this.

\begin{lemma}\label{lemma:localfamily}
Let $p\in \Gr^+(H_\RR)$ be such that $p\in \Gr^+(H_{\delta, \RR})$ for a \emph{unique} $\delta\in \Delta_\Ccal$. Then $p$ has a neighborhood $B$  in $Gr_+(H_{\delta, \RR})$ such that $B\cap \Gr^+(H_\RR)_{\Delta_\Ccal}$ supports a smooth family of Einstein manifolds.
\end{lemma}
\begin{proof}
Recall that in the commutative diagram ($\dagger$) at the end of Section 2, the period map $\Per$ is surjective on every connected component of $\teich$. 
The complex structures on $M$ parametrized  by $\teich$ for which $\delta$ is 
of Hodge type $(-1,-1)$ define a divisor $D_\delta$ of $\teich_{s}$ whose preimage in $\teich$ consists of nonseparated points.
Choose distinct points $\hat p_+, \hat p_-$  in $\teich_{\HK}$ that both lie over $p$, have distinct images $\tilde p_+, \tilde p_-$ in $ \teich$, but have the same image $p_s$ in $\teich_{s}$.  Then $\hat p_+$ and $\hat p_-$ lie in distinct K\"ahler cones, the values of their K\"ahler classes taking on $\delta$ opposite signs. Choose a small contractible neighborhood $U$ of $p_s$ in $\teich_s$ and denote by $\hat U_\pm$ the lift of $U$ in $\teich$  passing through $\hat p_\pm$ in $\teich$.  These two open subsets meet, and each gives  a  family of hyperk\"ahlerian manifolds over $U$. By part (iii) of Proposition \ref{prop:separation}  they are then isomorphic over $\hat U_+\cap \hat U_-$. We use such an isomorphism to turn this into
a family over the connected nonseparated space  $\hat U:=\hat U_+\cup \hat U_-$.  By taking all possible  K\"ahler Einstein metrics on the fibers of the family parametrized by $\hat U$ we then  obtain a family over  the preimage  of $\hat U$ in $\teich_\HK$. Forgetting the complex structure (but retaining the metric) then
makes this part of a twistor construction over an open subset $B$ of $\teich_\HH$. This open subset is as desired.
\end{proof}

\begin{proof}[Proof of Corollary \ref{cor:fine_einstein}]
Let us  write $G$ for the finite group $\aut_0(\Ccal)$. We have seen that if $g$ is an Einstein metric on $M$, then the group of isometries of $(M,g)$ that are isotopic to the identity is isomorphic to $G$; the choice of such an isomorphism  makes it an Einstein manifold with $G$-action. 

The set of points in  $\Gr^+(H_\RR)$ which are contained in at most one $\Gr^+_3(H_{\delta,\RR})$ with $\delta\in \Delta_\Ccal$, make up an open subset 
$\Gr^+(H_\RR)^\circ$ of $\Gr^+(H_\RR)$ whose complement is of codimension $6$. In particular, $\Gr^+(H_\RR)^\circ$ is $2$-connected (even  $4$-connected). 

With the help of Lemma \ref{lemma:localfamily} we can find  a \emph{Leray} covering 
$\Uscr:=\{\tilde U_\alpha\}_{\alpha\in A}$ of  $\Gr^+(H_\RR)^\circ$ by open subsets such that for every $\alpha\in A$, 
$U_\alpha:=\tilde U_\alpha\cap \Gr^+(H_\RR)_{\Delta_\Ccal}$ supports a family of $G$-Einstein manifolds $\Ncal_\alpha/U_\alpha$.  This means that every nonempty intersection of members of $\Uscr$ is contractible, so that a partition of unity subordinate to $\Uscr$ defines a homotopy equivalence from $\Gr^+(H_\RR)^\circ$ to the nerve of $\Uscr$. So for every abelian group $Z$, the natural map $\check{H}^\pt(\Uscr;Z)\to H^\pt(\Gr^+(H_\RR)^\circ;Z)$ 
is an isomorphism. Since $\Gr^+(H_\RR)^\circ$ is $2$-connected, the latter is zero in degree $2$.

If $(\alpha, \beta)\in A^2$ is a $1$-simplex of the nerve of $\Uscr$, then we have over  $U_{\alpha\beta}$ two such families. They are isomorphic and so there exists a 
$G$-equivariant isomorphism $f_{\alpha\beta}: \Ncal_\alpha|U_{\alpha\beta}\cong \Ncal_\beta|U_{\alpha\beta}$ over $U_{\alpha\beta}$. Then for every 
$2$-simplex  $(\alpha, \beta, \gamma)\in A^3$ of the nerve, $f_{\gamma\alpha}f_{\beta\gamma}f_{\alpha\beta}$ is an automorphism of $\Ncal_\alpha/U_{\alpha\beta\gamma}$. Since this automorphism preserves each $G$-orbit, it must be given by the action of a central element 
$c_{\alpha\beta\gamma}\in Z(G)$. It is clear that $(\alpha, \beta, \gamma)\mapsto c_{\alpha\beta\gamma}$ is then a $Z(G)$-valued \v{C}ech $2$-cocycle for $\Uscr$ and hence defines
an element of $\check{H}^2(\Uscr; Z(G))$. Since the latter is trivial,  
this cocycle is a coboundary: there exists a $Z(G)$-valued  \v{C}ech $1$-cochain $(\alpha, \beta)\mapsto c_{\alpha\beta}$ such that
$c_{\alpha\beta\gamma}=c_{\gamma\alpha}c_{\beta\gamma}c_{\alpha\beta}$. Upon replacing $f_{\alpha\beta}$ with $f_{\alpha\beta}c_{\alpha\beta}^{-1}$, we then 
arrange that $f_{\gamma\alpha}f_{\beta\gamma}f_{\alpha\beta}$ is always the identity, wherever defined. This implies that if we glue the families $\Ncal_\alpha/U_\alpha$
by means of these isomorphisms we obtain a well-defined family over $\Gr^+(H_\RR)_{\Delta_\Ccal}$.
\end{proof}

\begin{question}\label{question}
Corollary \ref{cor:fine_einstein} leads us to  the purely topological question whether the group homomorphism $\rho: \Mod(M)\to \Orth (q)$ has almost a section in the sense that  there exists a subgroup  $\tilde\Gamma$ of  $\Mod(M)$  such that $\rho|\tilde\Gamma$ is injective and $\rho(\tilde\Gamma)$ is of finite index in $\Orth (q)$. If such a $\tilde\Gamma$ exists, then by taking the  $\tilde\Gamma$-quotient  we obtain a genuine family over an arithmetic quotient of $\teich_{\HH}$. Otherwise, there is no such family.

Since  for a connected component $\Ccal$ of $\teich$, the image of the stabilizer $\Mod(M)_\Ccal$ in $\Orth (q)$ is of finite index, we can take $\tilde\Gamma\subset \Mod(M)_\Ccal$. So by Corollary \ref{cor:family}, the answer is yes when $\aut_0(\Ccal)$  is trivial. In general,  the group $\Mod(M)_\Ccal$ acts on the finite group $\aut_0(\Ccal)$ by inner automorphisms, so that we want to choose $\tilde\Gamma$ in the kernel of this action.
Then  $\tilde\Gamma$ will be an extension of $\Gamma$ by a central subgroup $Z:=\tilde\Gamma\cap\aut_0(\Ccal)$ and the question comes down to  whether  such a central extension is residually  finite.  Put differently, this central extension is given by an element of $H^2(\rho(\tilde\Gamma); Z)$ and the question is then whether this element dies when we  restrict it to some subgroup of $\rho(\tilde\Gamma)$ of finite index.  This is a nontrivial issue, for 
the algebraic universal cover of $\SO(q_\RR)$ is the Spin group $\spin(q_\RR)$, but the absolute metaplectic cover of the latter is still of order $2$ when $n\ge 3$ and it is not clear whether this cover contains arithmetic groups that are residually finite (see \cite{prasadrap}).
\end{question}

\begin{question}\label{question:structuregroup} 
A connected component $\Ccal$ of  $\teich_\HH$ was identified with an open subset of $\Gr^+_3(H_\RR)$,  and so it inherits from this 
a locally symmetric metric and hence a notion of geodesic interval. The 
twistor construction singles out geodesic intervals of a particular type: Let two elements of  $\Gr^+_3(H_\RR)$ be represented by 
the  $3$-planes $P_0$ and $P_1$ and assume that these  have a $2$-plane $\Pi$ in common. Recall that by our convention, $P_0$ and $P_1$
are naturally oriented and so an orientation of  $\Pi$ determines a ray $r_i$ in the  orthogonal complement of $\Pi$ in $P_i$. If we connect $r_0$ with $r_1$ in the orthogonal complement  of $\Pi$ in $P_0+P_1$ (in the obvious manner) by a path $\{ r_t\}_{t\in [0,1]}$, then the orthogonal complement $P_t$ of $r_t$  in $P_0+P_1$ traverses a geodesic segment  $[P_0,P_1]$  in $\Gr^+_3(H_\RR)$. Suppose now that $P_0$ represents an Einstein  metric $g_0$ on $M$. Then  $\Pi$ defines a member of the associated
twistor family and hence defines a complex structure on $M$ for which $g_0$ is K\"ahler-Einstein.  For the underlying complex manifold $X$, 
the family $r_t$ defines an interval in its (projectivized) K\"ahler cone,  hence gives a path of Einstein metrics \emph{on} $M$ that begins with $g_0$.
If this is part of a piecewise geodesic loop $(P_0, P_1, \dots , P_k)$ with $\dim (P_{i-1}\cap P_i)\ge 2$ and $P_0=P_k$, then
we also get a loop in $\Ccal$ (the most basic case is that of a small triangle  with $P_0, P_1, P_2$ having a line in common).
This means that the Einstein metric on $M$ that we end up with must differ from $g_0$ by an isotopy of $M$. 

What is the subgroup of $\diff^0(M)$ generated by such isotopies? Note that this group (which  is unique up to conjugacy in $\diff^0(M)$) may be regarded as the structure group of the universal bundle over $\Ccal$ . 
A recent theorem of Giansiracusa-Kupers-Tshishiku \cite{gkt} asserts that  for  a $K3$ surface $M$, the natural map $\diff^+(M)\to \Mod(M)$ does not split, not even over a subgroup of finite index. So for such surfaces this must be an infinite group
(\footnote{There is a similar question for the usual Teichm\"uller space: if $C$ is a closed Riemann surface of genus $\ge 2$, then a  closed loop in its  Teichm\"uller space consisting of piecewise Teichm\"uller geodesics defines a quasi-conformal self-homeomorphism of $C$ isotopic  to the identity. The subgroup generated by such  self-homeomorphisms is clearly the image of an increasing union 
of connected (finite dimensional) manifolds. What is this subgroup? We have been asking around for a while, but no-one seems to know.}).
\end{question}

\subsection*{The period map for the full cohomology}
In this subsection it is convenient to adopt the language of the theory of algebraic groups and in particular that of Shimura varieties.

The functor which assigns to any $\QQ$-algebra $R$, the subgroup $\SO (q_R)\subset \GL(H_R)$,  is represented by a  $\QQ$-algebraic group $\SOs_q$, so that for example $\SOs_q(\RR)=\SO(q_\RR)$. Although $\SO(q_\RR)$ has two connected components when $n>0$, 
as an algebraic group, $\SOs_q$ is connected.
We denote by  $\spins_q$ the algebraic universal cover $\SOs_q$. This is a semi-simple algebraic  group defined over $\QQ$ and 
$\spins_q(\RR)$ is the usual $\spin (q_\RR)$ (which however  is not simply connected for the Hausdorff topology: for $n\ge 3$ its fundamental group has order $2$)
and has $\Gr^+_3(H_\RR)$ as its symmetric space.  We identify the kernel of $\spins_q\to \SOs_q$ with $\mu_2=\{ \pm1\}$ and put
$\cspins_q:=\spins_q\times^{\mu_2}\GG_m$. This is a reductive algebraic  group over $\QQ$ that can be regarded as an extension of
$\SOs_q$ by $\GG_m$, but whose commutator subgroup is
$\spins_q$. 
It is clear that the action of $\spins_q\times \GG_m$ on 
 $H_\QQ$ for which $\spins_q$ acts via $\SOs_q$ and $t\in\GG_m$ as scalar multiplication with $t^{-2}$, factors through $\cspins_q$ 
and makes  $H_\QQ$ a $\QQ$-representation of $\cspins$.

Any  $z\in\per$  defines an embedding $\overline{j}_z: \Un(1)\hookrightarrow \SOs_q(\RR)$ that is given by clockwise rotation in the oriented plane $\Pi_z$ and as the identity in $\Pi_z^\perp$. Its preimage in $\spins_q(\RR)$ is a double (connected) cover in the sense that it yields a group homomorphism $j_z: \Un(1)\to \spins_q(\RR)$ whose square lifts $\overline{j}_z$. We may thus identify $\per$ with a distinguished conjugacy class of group monomorphisms $j_z: \Un(1)\to \spins_q(\RR)$. The preimage of the center in this new copy of $\Un(1)$ is $\mu_2$. The preimage 
of $\overline{j}_z$ under the projection $\cspins(\RR)\to \SOs_q(\RR)$ is of course a copy of $\Un (1)\times^{\mu_2}\RR^\times$, which is just a complicated way of writing $\CC^\times$, but regarded as the group of real points of a group defined over $\RR$. In other words, it is a copy of 
the Deligne torus $\Ss(\RR)$. Thus $z\in \per$ also determines  a group homomorphism  $J_z: \Ss(\RR)\to \spins_q(\RR)$. This identifies $\per$ with a conjugacy  class of such homomorphisms. The real tangent space of $\per$ at $z$ is $\Hom_\RR(\Pi_z, \Pi_z^\perp/\Pi_z)$. On this $\Ss(\RR)$ acts  via $\Pi_z$ (and induces the given complex structure on this tangent space) so that the Hodge numbers of  $T_z\per$ are $(1,-1)$ and $(-1,1)$.  For every  finite dimensional real representation $V$ of $\cspins$, the diagonal action of  $\cspins (\RR)$ on $V\times\per$ turns the trivial local system $\per$  with  fiber $V$ into a family $\Vs$ of real Hodge structures over $\per$. The natural maps $T_z\per\to \End_\RR(\Vs_z)$ are $\Ss(\RR)$-equivariant, so that  $\Vs$ satisfies  Griffiths transversality, hence  is a variation of Hodge structure. This makes $\per$ almost a Shimura variety as then any hyperplane section of $\per$ defined by a negative $\QQ$-linear form on $H_\QQ$  is a  Shimura variety (of orthogonal type) with respect to its $\cspins$-stabilizer. 

Let $g$ be an Einstein  metric on $M$ and denote the resulting Riemann manifold by $N$ as before. Then we have associated to $N$ an algebra of quaternions $\HH_N$ and a positive $3$-plane $P_N\subset H_\RR$ such that $\HH_N^\times$ acts on 
$P_N$ as the subgroup of $\cspins(\RR)$ which leaves $P^\perp$ pointwise fixed.  The embedding $\HH_N^\times\hookrightarrow \cspins(\RR)$ is 
then unique and takes the distinguished conjugacy class in $\Hom (\Ss(\RR), \HH_N^\times)$  to  the distinguished conjugacy class 
in $\Hom (\Ss(\RR), \cspins(\RR))$.
Let us refer to the image of such an $\HH_N^\times$  as \emph{twistor subgroup} of $\cspins(\RR)$. We thus recover a recent result due independently to Soldatenkov (Thm.\ 3.6 of \cite{soldatenkov}) and  Green-Kim-Laza-Robles (Thm.\ 4.1 of \cite{gklr}). 

\begin{corollary}[Soldatenkov,  Green-Kim-Laza-Robles]\label{cor:}
Let $\Ccal$ be a connected component of $\teich$ and  identify its separated quotient with $\per$. 
Then the associated variation of Hodge structure on the full cohomology $H^\pt(M; \QQ)$ over $\per$ is defined by a $\QQ$-representation of $\cspins(\RR)$  on  $H^\pt(M; \QQ)$.
\end{corollary}
\begin{proof}
We have seen that this is true when we restrict to a twistor family and the corresponding twistor subgroup. 
Each  $\eta\in H_\RR=H^2(M; \RR)$ which can appear as a K\"ahler class relative to some complex structure on $M$ defines a representation of $\sllie_2(\RR)$ on $H^\pt(M; \RR)$. It was proved in \cite{LL} that these generate a graded semisimple Lie algebra defined over $\QQ$, of which the degree zero part is the Lie algebra of $\cspins(\RR)$, except that  a generator of its center   acts on $H^k(M; \QQ)$ as multiplication with $2m-k$ rather that $-k$ (so that is a shift over $2m$). In view of  Remark  \ref{rem:lie}, this implies that  $H^\pt(M; \QQ)$ affords  a $\QQ$-representation of  $\cspins$ that is compatible with the twistor representations. 
Since the twistor subgroups generate $\cspins(\RR)$ and the union of the twistor families make up a dense subset of $\Ccal$, the assertion follows in general.
\end{proof}

\begin{remark}\label{rem:}
Kurnosov-Soldatenkov-Verbitsky \cite{ksv}  recently produced a Kuga-Satake construction for the rational cohomology of $M$.
This is accomplished by a  $\QQ$-representation of $\cspins$ of weight one whose exterior algebra contains the 
rational cohomology of $M$ as a graded $\cspins$-submodule.
\end{remark}

\end{document}